\documentclass[a4paper,12pt,reqno]{amsart}
\pdfoutput=1
\usepackage{amssymb,amsmath,amsthm,amsfonts,centernot}
\usepackage[colorlinks=true,linkcolor=blue,citecolor=blue,urlcolor=cyan, pdfpagelabels=false]{hyperref}

\usepackage{amsmath,tikz}
\usepackage{makecell}
\usepackage{dsfont}
\usetikzlibrary{matrix}
\usepackage{perpage}
\usepackage{makecell}
\usepackage{bm}
\MakePerPage{footnote}
\begin{document}
\numberwithin{equation}{section}
\newtheorem{thm}{Theorem}
\newtheorem{algo}{Algorithm}
\newtheorem{lem}{Lemma} 
\newtheorem{de}{Definition} 
\newtheorem{ex}{Example}
\newtheorem{pr}{Proposition} 
\newtheorem{claim}{Claim} 
\newtheorem{re}{Remark}
\newtheorem{co}{Corollary}
\newtheorem{conv}{Convention}
\newcommand{\di}{\hspace{1.5pt} \big|\hspace{1.5pt}}
\newcommand{\idi}{\hspace{.5pt}|\hspace{.5pt}}
\newcommand{\hs}{\hspace{1.3pt}}
\newcommand{\thmf}{Theorem~1.15$'$}
\newcommand{\ndi}{\centernot{\big|}}
\newcommand{\nidi}{\hspace{.5pt}\centernot{|}\hspace{.5pt}}
\newcommand{\lp}{\mbox{$\hspace{0.12em}\shortmid\hspace{-0.62em}\alpha$}}
\newcommand{\PQ}{\bb{P}^1(\bb{Q})}
\newcommand{\pmn}{\cl{P}_{M,N}}
\newcommand{\lcm}{\operatorname{lcm}}
\newcommand{\he}{holomorphic eta quotient\hspace*{2.5pt}}
\newcommand{\hes}{holomorphic eta quotients\hspace*{2.5pt}}
\newcommand{\defG}{Let $G\subset\GG$ be a subgroup that is conjugate to a finite index subgroup of $\G$. } 
\newcommand{\defg}{Let $G\subset\GG$ be a subgroup that is conjugate to a finite index subgroup of $\G$\hs\hs} 
\renewcommand{\phi}{\varphi}
\newcommand{\Z}{\bb{Z}}
\newcommand{\ZD}{\Z^{\D}}
\newcommand{\N}{\bb{N}}
\newcommand{\Q}{\bb{Q}}
\newcommand{\pii}{{{\pi}}}
\newcommand{\R}{\bb{R}}
\newcommand{\C}{\bb{C}}
\newcommand{\I}{\hs\cl{I}_{n,N}}
\newcommand{\SL}{\operatorname{SL}_2(\Z)}
\newcommand{\St}{\operatorname{Stab}}
\newcommand{\D}{\cl{D}_N}
\newcommand{\rh}{{{\boldsymbol\rho}}}
\newcommand{\bh}{{\cl{M}}} 
\newcommand{\lv}{\hyperlink{level}{{\text{level}}}\hspace*{2.5pt}}
\newcommand{\fct}{\hyperlink{factor}{{\text{factor}}}\hspace*{2.5pt}}
\newcommand{\q}{\hyperlink{q}{{\mathbin{q}}}}
\newcommand{\rd}{\hyperlink{redu}{{{\text{reducible}}}}\hspace*{2.5pt}}
\newcommand{\ird}{\hyperlink{irredu}{{{\text{irreducible}}}}\hspace*{2.5pt}}
\newcommand{\str}{\hyperlink{strong}{{{\text{strongly reducible}}}}\hspace*{2.5pt}}
\newcommand{\rdn}{\hyperlink{redon}{{{\text{reducible on}}}}\hspace*{2.5pt}}
\newcommand{\atl}{\hyperlink{atinv}{{\text{Atkin-Lehner involution}}}\hspace*{3.5pt}}
\newcommand{\atls}{\hyperlink{atinv}{{\text{Atkin-Lehner involutions}}}\hspace*{3.5pt}}
\newcommand{\T}{\mathrm{T}}
\renewcommand{\H}{\fr{H}}
\newcommand{\W}{\text{\calligra W}_n}
\newcommand{\GG}{\cl{G}}
\newcommand{\g}{\fr{g}}
\newcommand{\Gm}{\Gamma}
\newcommand{\Gmtl}{\widetilde{\Gamma}_\ell}
\newcommand{\gm}{\gamma}
\newcommand{\go}{\gamma_1}
\newcommand{\gmt}{\widetilde{\gamma}}
\newcommand{\gmdt}{\widetilde{\gamma}'}
\newcommand{\gmot}{\widetilde{\gamma}_1}
\newcommand{\gmdot}{{\widetilde{\gamma}}'_1}
\newcommand{\s}{\Large\text{{\calligra r}}\hspace{1.5pt}}
\newcommand{\ms}{m_{{{S}}}}
\newcommand{\nisim}{\centernot{\sim}}
\newcommand{\level}{\hyperlink{level}{{\text{level}}}}
\newcommand{\Redcon}{the \hyperlink{red}{\text{Reducibility~Conjecture}}}
\newcommand{\ReDcon}{The \hyperlink{red}{\text{Reducibility~Conjecture}}}
\newtheorem*{ThmA}{Theorem A}
\newtheorem{Conj}{Conjecture}
\newtheorem*{ThmB}{Theorem B}
\newtheorem*{ThmC}{Theorem C}
\newtheorem*{lmB}{Lemma B}
\newtheorem*{CorA}{Corollary A}
\newtheorem*{CorB}{Corollary B}
\newtheorem*{CorC}{Corollary C}
\newcommand{\Conred}{Conjecture~$1'$}
\newcommand{\effth}{Theorem~$\ref{27.11.2015}'$}
\newcommand{\Conredd}{Conjecture~$1''$}
\newcommand{\Conreddd}{Conjecture~$1'''$}
\newcommand{\Conired}{Conjecture~$2'$}
\newtheorem*{pro}{\textnormal{\textit{Proof of Lemma~\ref{27.11.2015.1}}}}
\newtheorem*{cau}{Caution}
\newtheorem{thrmm}{Theorem}[section]
\newtheorem*{thmA}{Theorem~A}
\newtheorem*{corA}{Corollary~A}
\newtheorem*{corB}{Corollary~B}
\newtheorem*{corC}{Corollary~C}
\newtheorem{no}{Notation}
\renewcommand{\thefootnote}{\fnsymbol{footnote}}
\newtheorem{oq}{Open Question}
\newtheorem{hy}{Hypothesis} 
\newtheorem{expl}{Example}
\newcommand\ileg[2]{\bigl(\frac{#1}{#2}\bigr)}
\newcommand\leg[2]{\Bigl(\frac{#1}{#2}\Bigr)}
\newcommand{\e}{\eta}
\newcommand{\sgn}{\operatorname{sgn}}
\newcommand{\bb}{\mathbb}
\newtheorem*{conred}{Conjecture~\ref{con1}$\mathbf{'}$}
\newtheorem*{conredd}{Conjecture~\ref{con1}$\mathbf{''}$}
\newtheorem*{conreddd}{Conjecture~\ref{con1}$\mathbf{'''}$}
\newtheorem*{conired}{Conjecture~\ref{19.1Aug}$\mathbf{'}$}
\newtheorem*{efth}{Theorem~\ref{27.11.2015}$\mathbf{'}$}
\newtheorem*{eflem}{Theorem~\ref{15May15}.$(b)\mathbf{'}$}
\newtheorem*{procl}{\textnormal{\textit{Proof}}}
\newtheorem*{thmff}{Theorem~\ref{7.4Jul}$\mathbf{'}$}
\newtheorem*{coo}{Corollary~\ref{17Aug}$\mathbf{'}$}
\newcommand{\cooo}{Corollary~\ref{17Aug}$'$}
\newtheorem*{cotw}{Corollary~\ref{17.1Aug}$\mathbf{'}$}
\newcommand{\cotww}{Corollary~\ref{17.1Aug}$'$}
\newtheorem*{cothr}{Corollary~\ref{17.2Aug}$\mathbf{'}$}
\newcommand{\cothre}{Corollary~\ref{17.2Aug}$'$}
\newtheorem*{cne}{Corollary~\ref{15.5Aug}$\mathbf{'}$}
\newcommand{\cnew}{Corollary~\ref{15.5Aug}$'$\hspace{3.5pt}}
\newcommand{\fr}{\mathfrak}
\newcommand{\cl}{\mathcal}
\newcommand{\rad}{\mathrm{rad}}
\newcommand{\ord}{\operatorname{ord}}
\newcommand{\m}{\setminus}
\newcommand{\G}{\Gamma_1}
\newcommand{\GN}{\Gamma_0(N)}
\newcommand{\X}{\widetilde{X}}
\renewcommand{\P}{{\textup{p}}} 
\newcommand{\al}{{\hs\operatorname{al}}}
\newcommand{\p}{p_\text{\tiny (\textit{N})}}
\newcommand{\pN}{p_\text{\tiny\textit{N}}}
\newcommand{\bt}{\mbox{$\raisebox{-0.59ex}
  {${{l}}$}\hspace{-0.215em}\beta\hspace{-0.88em}\raisebox{-0.98ex}{\scalebox{2}
  {$\color{white}.$}}\hspace{-0.416em}\raisebox{+0.88ex}
  {$\color{white}.$}\hspace{0.46em}$}{}}
  \newcommand{\un}{\hs\underline{\hspace{5pt}}\hs}
\newcommand{\U}{u_\textit{\tiny N}}
\newcommand{\Upr}{u_{\textit{\tiny N}^\prime}}
\newcommand{\Up}{u_{\textit{\tiny p}^\textit{\tiny e}}}
\newcommand{\Un}{u_{\textit{\tiny p}_\textit{\tiny 1}^{\textit{\tiny e}_\textit{\tiny 1}}}}
\newcommand{\Um}{u_{\textit{\tiny p}_\textit{\tiny m}^{\textit{\tiny e}_\textit{\tiny m}}}}
\newcommand{\Ut}{u_{\text{\tiny 2}^\textit{\tiny a}}}
\newcommand{\At}{A_{\text{\tiny 2}^\textit{\tiny a}}}
\newcommand{\Uh}{u_{\text{\tiny 3}^\textit{\tiny b}}}
\newcommand{\Ah}{A_{\text{\tiny 3}^\textit{\tiny b}}}
\newcommand{\Uprl}{u_{\textit{\tiny N}_1}}
\newcommand{\Uprlm}{u_{\textit{\tiny N}_i}}
\newcommand{\UM}{u_\textit{\tiny M}}
\newcommand{\UMp}{u_{\textit{\tiny M}_1}}
\newcommand{\w}{\omega_\textit{\tiny N}}
\newcommand{\wm}{\omega_\textit{\tiny M}}
\newcommand{\wa}{\omega_{\text{\tiny N}_\textit{\tiny a}}}
\newcommand{\wma}{\omega_{\text{\tiny M}_\textit{\tiny a}}}
\renewcommand{\P}{{\textup{p}}} 
\newcommand{\st}[1]{\substack{#1}}


\title[Diophantine approximation with sums of two squares]{Diophantine approximation with sums of two squares}

\author{Stephan Baier}
\address{Stephan Baier,
Ramakrishna Mission Vivekananda Educational and Research Institute, Department of Mathematics, G. T. Road, PO Belur Math, Howrah, West Bengal 711202, India}
\email{stephanbaier2017@gmail.com}
\author{Habibur Rahaman}
\address{Habibur Rahaman, Indian Institute of Science Education \& Research Kolkata,
Department of Mathematics and Statistics, Mohanpur, West Bengal 741246, India}
\email{hr21rs044@iiserkol.ac.in}


\maketitle
 
\begin{abstract}
For any given positive definite binary quadratic form $Q$ with integer coefficients, we establish two results on Diophantine approximation with integers represented by $Q$. Firstly, we show that for every irrational number $\alpha$, there exist infinitely many positive integers $n$ represented by $Q$ and satisfying $\|\alpha n\|<n^{-(1/2-\varepsilon)}$ for any fixed but arbitrarily small $\varepsilon>0$. This is an easy consequence of a result by Cook on small fractional parts of diagonal quadratic forms. Secondly, we give a quantitative version with a lower bound of this result when the exponent $1/2-\varepsilon$ is replaced by any fixed $\gamma<3/7$.  To this end, we use the Voronoi summation formula and a bound for bilinear forms with Kloosterman sums to fixed moduli by Kerr, Shparlinski, Wu and Xi.    
\end{abstract}

 \tableofcontents
 
\section{Introduction and main results}
Throughout this article, we denote by $\|z\|$ the distance of $z\in \mathbb{R}$ to the nearest integer. As usual, $\varepsilon$ stands for an arbitrarily small positive number. Moreover, we write
$$
e(x)=e^{2\pi i x} \quad \mbox{and} \quad e_k(x)=e\left(\frac{x}{k}\right) \quad \mbox{for } x\in\mathbb{R} \mbox{ and } k\in \mathbb{N}.
$$

A basic theorem in Diophantine approximation is the Dirichlet approximation theorem, which implies that for every irrational number $\alpha$ there are infinitely many positive integers $n$ such that $\|\alpha n\|<n^{-1}$, where $\|x\|$ is the distance of $x$ from the nearest integer. Beyond this classical result, mathematicians have explored obtaining good approximations when $n$ is restricted to sparse subsets of the set of positive integers. Numerous results exist in this direction. Here we mention only a few of them. For example, when $n$ varies over the set of primes, the best known approximation result states that $\|\alpha n\|<n^{-1/3+\varepsilon}$ holds for infinitely many primes $n$. This is due to Matomäki \cite{Matomaki}. 
In the case when $n$ varies over the set of square-free positive integers, it was established by Heath-Brown \cite{H-B} that $\|\alpha n\|<n^{-2/3+\varepsilon}$ holds for infinitely many squarefree $n\in \mathbb{N}$. In this article, we consider the case where $n$ varies over positive integers which are represented by any given positive definite binary quadratic form with integer coefficients. 

Let 
$$
Q(x,y)=a_1x^2+b_1xy+c_1y^2
$$ 
be a positive definite integral binary quadratic form (PBQF). 
In particular, the discriminant $b_1^2-4a_1c_1$ of $Q$ is negative. We set 
$$
\Delta:=|b_1^2-4a_1c_1|=4a_1c_1-b_1^2.
$$
Let 
\begin{align}\label{def_A_Q}
    \mathcal{A}_Q:=\{n\in\mathbb{N}: Q(x,y)=n \text{ for some } x,y\in\mathbb{Z}\}.
\end{align}
 Let $\alpha$ be any fixed irrational number. We establish the following result.
 \begin{thm}\label{Theorem_1/2_exponent}
     There are infinitely many integers $n\in\mathcal{A}_Q$ such that 
     \begin{align}\label{1/2_exponent}
         \|\alpha n\|<n^{-1/2+\varepsilon},
     \end{align}
     for any fixed but arbitrarily small $\varepsilon>0$.
 \end{thm}
Theorem~\ref{Theorem_1/2_exponent} improves the following theorem of Danicic for the binary quadratic case.
\begin{thm} \cite[Main Theorem]{Danicic}\label{Danicic}
    Let 
    \begin{align*}
        Q(x_1, x_2, \ldots, x_s)=\sum_{i=1}^{s}\sum_{j=1}^{s}\alpha_{ij}x_ix_j, \ \ \ \ (\alpha_{ij}=\alpha_{ji}),
    \end{align*}
    be a real quadratic form in $s$ variables $x_1, x_2,\ldots,x_s$. If $N>1$ and $\varepsilon>0$, then there exist integers $x_1, x_2, \ldots, x_s$, not all zero, satisfying 
    \begin{align*}
        \max\limits_{1\le i\le s}  |x_i|\leq N \ \  \text{ and }\ \ \|Q(x_1,x_2,\ldots,x_s)\|<C(\varepsilon, s) N^{-s/(s+1) +\varepsilon}.
    \end{align*}
\end{thm}
For the case $s=2$, taking $n=Q(x_1, x_2)$, this gives the exponent $1/3$ in \eqref{1/2_exponent} instead of $1/2$. Though $1/2$ is a strong improvement over $1/3$, Theorem~\ref{Theorem_1/2_exponent} can be deduced easily from the following theorem of Cook. It seems this has not been observed in the literature so far. 

 \begin{thm}\cite[Theorem~1]{Cook}\label{thm_Cook}
     Let $k\geq 2$ and $K=2^{k-1}.$ Let $f(\vec{x})=\theta_1x_1^k+\theta_2x_2^k+\cdots+\theta_sx_s^k$ be an additive form of degree $k$, in $s$ variables with $1\leq s\leq K$, where $\theta_i$ are any real numbers and $\vec{x}=(x_1, x_2,\cdots,x_s).$ Then for every $\epsilon>0, N>1$, there are integers $x_1, x_2, \ldots, x_s$, not all zero, such that 
     \begin{align*}
         \max\limits_{1\le i\le s} |x_i|\leq N\  \text{ \ and \ } \ \|f(\vec{x})\|<C(\varepsilon, k) N^{-s/K+\varepsilon}.
     \end{align*}
 \end{thm}
\hspace{-0.5cm} {\bf Proof of Theorem 1.} Let $\alpha$ be any given irrational number. We have 
 \begin{align*}
     \alpha Q(x,2a_1y)&=\alpha (a_1x^2+ 2a_1b_1xy+4a_1^2c_1y^2)\\
     &=\alpha a_1 \left((x+ b_1 y)^2+ (4a_1c_1-b_1^2)y^2\right)\\
     &=f(x_1,y_1), 
 \end{align*}
 where $f(x_1,y_1)=\theta_1 x_1^2+\theta_2y_1^2,$ with $\theta_1:=\alpha a_1, \theta_2:=\alpha(4a_1c_1-b_1^2), x_1:=x+b_1y$ and $y_1:=y$. Taking $k=s=K=2$ in Theorem~\ref{thm_Cook}, for any $N>1$ and $\varepsilon>0$, the inequality 
 \begin{align*}
     \|\alpha Q(x, 2a_1y)\|= \|f(x_1, y_1)\|< D_1(\varepsilon,Q)N^{-1+\varepsilon}
 \end{align*}
has a solution $(x,y)\in \mathbb{Z}^2$, where $x, y\le N$ and $D_1(\varepsilon,Q)>0$ is a suitable constant depending only on $\varepsilon$ and $Q$. Now taking $n=Q(x, 2a_1y)$ gives 
 \begin{align*}
     \|\alpha n\|<D_1(\varepsilon,Q)N^{-1+\varepsilon}\le  D_2(\varepsilon,Q)n^{-1/2+\varepsilon/2}
 \end{align*}
for another constant $D_2(\varepsilon,Q)>0$ since $Q$ is positive definite. Note that $\|\alpha n\|<D_1(\varepsilon,Q)N^{-1+\varepsilon}$ does not hold for arbitrarily large $N$ because $\alpha$ is irrational. Hence, there is an infinite sequence of natural numbers $n$ such that
$$
\|\alpha n\|<D_2(\varepsilon,Q)n^{-1/2+\varepsilon/2},
$$
and the result of Theorem 1 follows. \qed
 
\vspace{0.2cm}
Danicic \cite{Danicic} used lattices to establish Theorem \ref{Danicic}, and Cook \cite{Cook} used bounds for quadratic Weyl sums to establish Theorem \ref{thm_Cook}. Both methods give only the existence of infinitely many good approximations but no bounds on the number of them. The main part of this article consists of establishing a lower bound which holds for certain ranges of $n$ depending on good rational approximations of $\alpha$. This comes at the cost of a weaker exponent of $3/7$ in place of $1/2$, which still beats the exponent $1/3$ coming from Theorem \ref{Danicic}. Our method uses the Voronoi summation formula and bounds for bilinear sums with Kloosterman sums to a fixed modulus due to Kerr, Shparlinski, Wu and Xi \cite{Kerr}. We will deduce our result from the following asymptotic formula for the number of representations by $Q$ in unions of residue classes, which is of independent interest. 

 \begin{thm}\label{main_thm} Suppose that $q\in \mathbb{N}$ and $a\in \mathbb{Z}$ such that $(2\Delta a,q)=1$. Let $\varepsilon>0$ and set $L:=q^\beta$ and  $X:=Lq=q^{1+\beta}$ with $2/5+3\varepsilon\le \beta<1$. Let $w: \mathbb{R}\rightarrow \mathbb{R}_{\ge 0}$ be a smooth function which is compactly supported in $[X, 2X]$ with $w(x)=1$ in $[5X/4,7X/4]$ such that for any integer $j\geq 1$, the $j^{th}$ derivative satisfies
    \begin{align*}
        w^{(j)}(x)\ll_j X^{-j}.
    \end{align*}
    Define \begin{equation} \label{Sdef}
    S:=\sum\limits_{\substack{|b|\le L\\(b,q)=1}} \sum\limits_{\substack{n\in\mathbb{Z}\\ na\equiv b\bmod{q}}} r_Q(n)w(n), 
\end{equation}
where $r_Q(n)$ is the number of integral representations of $n$ by the PBQF $Q(x,y)$.
Then we have 
\begin{align*}
    S=  \frac{4\pi \varphi(q)L}{\Delta^{1/2}q^2}\prod_{\substack{p|q\\p \text{ prime }}}\left(1-\frac{G_Q(p,1)}{p^2}\right) \int\limits_{0}^{\infty}w(x)dx \cdot\left(1+o(1)\right)
\end{align*}
as $q\rightarrow\infty$, where 
$$
G_Q(q,h):=\sum_{x,y\bmod q}e_q(hQ(x,y)).
$$
Moreover, $S$ satisfies the lower and upper bounds 
$$
\frac{q^{2\beta}}{(\log\log q)^2}\ll S\ll q^{2\beta}.
$$
\end{thm}

We will use Theorem \ref{main_thm} above to deduce the following lower bound for the number of $n$'s represented by $Q$ satisfying $\|n\alpha\|<C_1n^{-\gamma}$ for some constant $C_1$ and exponent $\gamma$ in a suitable range. This lower bound is close to the expected bound. 

\begin{co} \label{main_thm 2} Suppose that $q\in \mathbb{N}$ and $a\in \mathbb{Z}$ satisfy $(2\Delta a,q)=1$ and the inequality  
$$
\left|\alpha-\frac{a}{q}\right|< \frac{24\Delta^2}{q^2}.
$$
Fix $\varepsilon,\eta>0$ and assume that $2/5+3\varepsilon\le \beta<1$. Set $X:=q^{1+\beta}$ and $\gamma:=
(1-\beta)/(1+\beta)$. Then
$$
\sharp \{n\in \mathcal{A}_Q: n\le 2X, \ \|n\alpha\|<C_1n^{-\gamma}\} \gg_{\varepsilon,\eta} X^{1-\gamma-C_2/(\log\log X)}   
$$
for suitable constants $C_1,C_2>0$. Moreover, there are infinitely many rational numbers $a/q$ satisfying the conditions above. 
\end{co}

\begin{re} We note that $(1-2/5)/(1+2/5)=3/7$, and thus Corollary \ref{main_thm 2} recovers Theorem 1 with a weaker exponent of $3/7$ in place of $1/2$, but coming with a lower bound for the number of $n$'s satisfying $\|n\alpha\|<n^{-3/7+\varepsilon}$. 
\end{re}

In an earlier version of this article, we referred to \cite[Theorem~3]{Balog} by Balog and Perelli who stated that the above Theorem \ref{main_thm} holds for the particular form $Q(x,y)=x^2+y^2$. They did not provide a proof, merely noting that this result follows by using their method developed in \cite{Balog} to investigate Diophantine approximation with squarefree integers. For this problem, they established an exponent of $1/2-\varepsilon$ indeed. However, analyzing their method, the authors of this article only obtained an exponent of $1/3-\varepsilon$ in place of $1/2-\varepsilon$, which resembles Danicic's result for the particular form $Q(x,y)=x^2+y^2$. We would like to thank Professor Glyn Harman for making us aware of the results of Cook and Danicic mentioned above. Below, we add some details on our analysis of the method used by Balog and Perelli.

They started in a similar way as we do here, converting the Diophantine condition $\|n\alpha\|<\delta$ into a condition of $n$ lying in a union of residue classes modulo $q$. Then they proceeded by detecting the squarefreeness of $n$ using the convolution identity $\mu^2(n)=\sum\limits_{t^2|n} \mu(t)$ and the congruence condition using the orthogonality relations for Dirichlet characters. After isolating a main term, they were left with character sums which were treated using H\"older's inequality and moment bounds for character sums. This yielded an exponent of $1/2-\varepsilon$ for the relevant Diophantine approximation problem. To handle Diophantine approximation with denominators represented by sums of two squares following their approach, one may use the convolution identity $r_2(n)=4\sum\limits_{d|n} \chi_4(d)$ with $r_2(n)$ being the number of representations of $n$ as a sum of two squares and $\chi_4$ the non-trivial character modulo 4 . However, here we have a divisor condition of $d|n$ in place of $t^2|n$. As a consequence (even if one uses the hyperbola trick), one gets only an exponent of $1/3-\varepsilon$ instead of $1/2-\varepsilon$ in the Diophantine approximation problem for sums of two squares. We also note that their method does not extend to general quadratic forms $Q$ because of the lack of a simple formula for the number of representations of $n$ in the form $n=Q(x,y)$.   

\section{Preliminaries}
We may write the quadratic form $Q(x,y)$ as 
$$Q(x,y)=\frac{1}{2}(x,y) A (x,y)^T,$$
where $A$ is an integral symmetric $2\times 2$ matrix with even diagonal elements, and $(x,y)^T$ is the column vector of the indeterminates. Then the adjoint matrix $A^{\dagger}$ of $A$ is defined by $A^\dagger A=\det(A) I_2$, where $\det(A)=\Delta$ is the determinant of the matrix $A$ and $I_2$ is the $2\times 2$ identity matrix. Clearly, $A^\dagger$ has also integer entries.

Our main tools are the Voronoi summation formula for quadratic forms and bounds for bilinear sums with Kloosterman sums to a fixed modulus. To state the Voronoi summation formula in this context, let us first set up some notations. 
For a positive integer $q$ and an integer $h$, we define the Gauss sum associated to the form $Q$ as 
\begin{align}\label{def_G(k,h)}
    G_Q(q,h):=\sum_{x,y\bmod q}e_q(hQ(x,y)).
\end{align}
The following lemma gives an evaluation of $G_Q(q,h)$ if $(2\Delta h,q)=1$.
\begin{lem}\label{Lem_G_Q(k,h,a)}
Suppose that $Q=\frac{1}{2}(x,y) A (x,y)^T$ is a PBQF with discriminant $\Delta$, and let $q, h$ be integers such that $(2\Delta h,q)=1$. 
Then we have
\begin{align}\label{formula_G(k,h)}
    G_Q(q,h)=\left(\frac{h}{N(q,Q)}\right)G_Q(q,1),
\end{align}
where $\left(\frac{h}{N(q,Q)}\right)$ is the Kronecker symbol with 
\begin{align}\label{def_N(k,Q)}
    N(q,Q):=\sharp\{\vec{x}(\bmod q): A\vec{x}\equiv\vec{0} (\bmod q)\}
\end{align}
and $\vec{x}=(x,y)^T.$
\end{lem}
\begin{proof}
    See \cite[Theorem~1]{Smith} or \cite[Theorem~1]{Jutila} for the particular case when $(2\Delta h, q)=1$.
\end{proof} 
We also note the following bound for the Gauss sum $G_Q(q,h)$ as it will be required later.
\begin{lem}\label{lem_bound_G_q(q,h)}
    For any $q\in \mathbb{N}$ and $h\in \mathbb{Z}$ with $(q,h)=1$, we have 
    \begin{align*}
        |G_Q(q,h)|\leq \left(N(q, Q)\right)^2 q,
    \end{align*}
where $N(q,Q)$ is given by \eqref{def_N(k,Q)}.
\end{lem}
\begin{proof}
This bound follows by diagonalizing the quadratic form, writing the resulting sum as a product of two one variable Gauss sums and then using the standard bound for them (see \cite[Lemma~2]{Smith}, or, \cite[Lemma~1]{Smith_1}). 
\end{proof}
Let $A_1=(s_i\delta_{ij})$ be the Smith Normal form of $A^\dagger$, i.e., $A_1$ is an integral diagonal matrix of the form 
\begin{align*}
    A_1=UA^\dagger V,
\end{align*}
where $U,V\in\SL$. Define further, 
\begin{align*}
B:=(b_i\delta_{ij})\quad \mbox{and} \quad  A^*:=BV^TA^\dagger VB.
\end{align*}
Let $Q^*$ be the quadratic form associated to the matrix $A^*$ and define 
\begin{align}\label{def_tilde{r}}
    \tilde{r}_{Q^*}(n,q,h):=q^{-1}\sum_{\substack{Q^*(\vec{x})=n}}\ \sum_{\vec{y}(\bmod q)}e_q(hQ(\vec{y})).
\end{align}
Now we are ready to state the Voronoi summation formula for the quadratic form $Q$.
\begin{lem}[Voronoi summation formula]\label{Voronoi}
    Let $f(n)$ be a continuously differentiable function on $\mathbb{R}$ with compact support in $\mathbb{R}_{\ge 0}$. Then under the conditions in Lemma \ref{Lem_G_Q(k,h,a)}, we have 
\begin{align*}
        &\sum_{n\in\mathbb{Z}}r_Q(n)e_q(nh)f(n)=2\pi \Delta^{-1/2}q^{-2}G_Q(q,h)\int\limits_0^{\infty} f(x) dx \\
        &+2\pi (q\sqrt{\Delta})^{-1}\sum_{n\in\mathbb{Z}}\tilde{r}_{Q^*}(n,q,h)e_{q}(-\overline{h}\overline{\Delta}n)\int\limits_0^{\infty} J_0\left(4\pi\frac{\sqrt{nx/\Delta}}{q}\right)f(x)dx,
    \end{align*}
where $J_0(x)$ is the standard Bessel function and $\overline{\Delta}$ is a multiplicative inverse of $\Delta$ modulo $q$, i.e. $\overline{\Delta}\Delta\equiv 1\bmod{q}$.  
\end{lem}
\begin{proof}
   See the proof of \cite[Equation (28)]{Jutila} for the special case when $(2\Delta h,q)=1$. 
\end{proof}
Our second main tool is the following theorem due Kerr et al. which estimates special bilinear forms with Kloosterman sums to a fixed modulus.
\begin{lem}\label{lem_bound_aver_Kloos}
Let $q$ be any positive integer and $a$ be any integer with $d=\gcd(a,q)$. 
    Suppose $I,J$ be some finite intervals of natural numbers with $|I|=M, |J|=N$ and let $\bm{\alpha}=\{\alpha_m\}$ be a sequence of complex numbers supported in $I$. Define 
    \begin{align*}
        S_{q,a}(\bm{\alpha}):=\sum_{m\in I}\sum_{n\in J}\alpha_m  S(m, an; q),
    \end{align*}
    where $S(m,n;q)$ denotes the Kloosterman sum 
    \begin{align*}
        S(m,n;q):=\sum_{\substack{1\leq x\leq q-1\\(x,q)=1}}e\left(\frac{mx+n\overline{x}}{q}\right).
    \end{align*}
    Then we have 
    \begin{align} \label{bibound}
        S_{q,a}(\bm{\alpha})\ll \left(\sum_{m\in I}|\alpha_m|^2\right)^{1/2}M^{1/2}Nq^{1/2+o(1)}S_1(M, N, q,d),
    \end{align}
    where one can take $S_1(M, N, q,d)$ freely among 
    \begin{align}
        M^{-1/4}N^{-1}q^{1/2}d^{-1/4}+M^{-1/2}N^{-1}q^{1/2}+N^{-1/2},\nonumber\\
        M^{-1/2}(N^{-3/4}q^{1/2}+d^{1/2})+N^{-1/2},\label{relevantbound}\\
        M^{-1/2}(N^{-1}q^{1/2}+(qd)^{1/4})+N^{-1/2}.\nonumber
        \end{align}
\end{lem}
\begin{proof}
    See the proof of \cite[Theorem~2.1]{Kerr}.
\end{proof}
The following Dirichlet approximation type lemma will allow us to choose the modulus $q$ coprime to $2\Delta$ (in fact coprime to any given non-zero integer). 
\begin{lem}\label{lem_Dirichlet_approx}
    Let $d\in \mathbb{N}$ and $\alpha\in \mathbb{R}\setminus\mathbb{Q}$. Then there exist infinitely many pairs $(b,r)\in \mathbb{Z}\times \mathbb{N}$ such that
$$
(r,bd)=1\quad \mbox{and} \quad \left|\alpha-\frac{b}{r}\right|\le \frac{6d^2}{r^2}. 
$$
\end{lem}
\begin{proof}
Fix $d\in \mathbb{N}$ and $\alpha\in \mathbb{R}\setminus \mathbb{Q}$. We aim to find good approximations of $\alpha$
by fractions $b/r$ with $(r,bd)=1$. Take one of the infinitely many approximations of $\alpha$ by $a/q$ with $(a,q)=1$ and 
\begin{equation} \label{1}
\left|\alpha-\frac{a}{q}\right|\le \frac{1}{q^2}.
\end{equation}
Now we aim to find integers $b$ and $r$ in such a way that $r$ lies approximately in the range of $q$, 
\begin{equation}\label{2}
\frac{a}{q}-\frac{b}{r}= \frac{f}{qr}
\end{equation}
with $f\ge 0$ small and $(r,bd)=1$. To ensure the coprimality, we simply demand that
\begin{equation} \label{3}
r\equiv 1\bmod{bd}.
\end{equation}
Let 
\begin{equation*}\label{4}
(d,q)=e.
\end{equation*}
For our method to work out, we take $f\ge 0$ as small as possible so that 
\begin{equation} \label{5}
f\equiv a\bmod{e}.
\end{equation}
We note that 
\begin{equation} \label{6}
0\le f\le e\le d.
\end{equation}
The equation \eqref{2} holds if 
\begin{equation} \label{7}
ar\equiv f\bmod{q}
\end{equation}
and 
\begin{equation*} \label{8}
b=\frac{ar-f}{q}.
\end{equation*}
The congruence \eqref{3} can be written in the form
\begin{equation} \label{9}
r=1+kbd, 
\end{equation}
where $k\in \mathbb{Z}$. Plugging this into \eqref{7} gives
\begin{equation*} \label{10}
a(1+kbd)\equiv f\bmod{q},
\end{equation*}
which is equivalent to 
\begin{equation*} \label{11}
kbd\equiv \overline{a}(f-a)\bmod{q},
\end{equation*}
where $a$ is a multiplicative inverse of $a$ modulo $q$. Since we assumed \eqref{5}, this congruence is equivalent to
\begin{equation} \label{12}
kb\cdot \frac{d}{e}\equiv \overline{a}\cdot \frac{f-a}{e} \bmod{\frac{q}{e}}.
\end{equation}
Write 
\begin{equation*} \label{13}
d_1:=\frac{d}{e}, \quad f_1:=\overline{a}\cdot \frac{f-a}{e}, \quad q_1:=\frac{q}{e}.
\end{equation*}
Then $(d_1,q_1)=1$, and \eqref{12} turns into
\begin{equation*} \label{14}
kbd_1\equiv f_1\bmod{q_1}. 
\end{equation*}
This can be solved for $2kb$, getting
\begin{equation*} \label{15}
kb\equiv f_1\overline{d_1} \bmod{q_1}, 
\end{equation*}
where $d_1$ is a multiplicative inverse of $d_1$ modulo $q_1$.
This may be written in the form
\begin{equation*} \label{16}
kb=c+q_1l,
\end{equation*}
where $l\in \mathbb{Z}$ and 
\begin{equation*} \label{17}
-\frac{q_1}{2}< c\le \frac{q_1}{2}\quad \mbox{ such that } \quad c\equiv f_1\overline{d_1}\bmod{q_1}.
\end{equation*}
Combining this with \eqref{9}, we obtain
\begin{equation*} \label{18}
r=1+d\left(c+q_1l\right).
\end{equation*}
To ensure that $r$ is of a comparable size as $q$, we choose $l=1$ and hence
\begin{equation*} \label{18'}
r=1+d\left(c+q_1\right).
\end{equation*}
In this case,
\begin{equation*} \label{19}
\frac{q}{2}\le \frac{dq_1}{2} \le r\le 1+\frac{3dq_1}{2}\le 1+\frac{3dq}{2}\le 2dq
\end{equation*}
if $q\ge 2$. From the upper bound above, we have
\begin{equation} \label{20}
q\ge \frac{r}{2d}. 
\end{equation}
Combining \eqref{1}, \eqref{2}, \eqref{6} and \eqref{20}, we obtain
\begin{equation*} 
\left|\alpha-\frac{b}{r}\right|\le \left|\alpha-\frac{a}{q}\right|+\left|\frac{a}{q}-\frac{b}{r}\right| \le \frac{1}{q^2}+\frac{f}{qr}\le \frac{1}{(r/(2d))^2}+\frac{d}{(r/(2d))\cdot r}=\frac{6d^2}{r^2}.
\end{equation*}  
\end{proof}
\section{Proof of Corollary \ref{main_thm 2}}
We start the proof of Corollary \ref{main_thm 2} by the following lemma providing a bound for $r_Q(n)$.
\begin{lem}\label{lem_bound_r_Q(n)} For any given positive integer $n$, let $r_Q(n)$ be the number of integral representations of $n$ by the PBQF $Q(x,y).$ Then we have 
    \begin{align*}
        r_Q(n)\ll_Q n^{c/\log\log n},
    \end{align*}
    for some $c>0.$
\end{lem}
\begin{proof}
Let $K:=\mathbb{Q}(\sqrt{-\Delta})$ and $\mathcal{O}_K$ be the ring of algebraic integers in $K$. Then $4a_1Q(x,y)$ factors over $\mathcal{O}_K$ in the form 
\begin{equation*}
\begin{split}
4a_1Q(x,y)=& 4a_1(a_1x^2+b_1xy+c_1y^2)\\
= & (2a_1x+b_1y-\sqrt{-\Delta}y)(2a_1x+b_1y+\sqrt{-\Delta}y).
\end{split}
\end{equation*}
Hence if $n=Q(x,y)$ with $x,y\in \mathbb{Z}$, then 
$$
4a_1n=\mathcal{N}\left(\mathfrak{a}\right),
$$
where $\mathcal{N}\left(\mathfrak{a}\right)$ is the norm of the principal ideal $\mathfrak{a}\subseteq \mathcal{O}_K$ given by 
$$
\mathfrak{a}=(2a_1x+b_1y+\sqrt{-\Delta}y).
$$
If $n=Q(x,y)$ is a representation of $n$ by $Q$, then we say that $\mathfrak{a}$ given above is the principal ideal belonging to this representation. We note that every principal ideal $\mathfrak{a}\subseteq \mathcal{O}_K$ belongs to at most $\omega\le 6$ representations of $n$ by $Q$, where $\omega$ is the number of units in $\mathcal{O}_K$.  Now for $m\in \mathbb{N}$ let $f(m)$ be the number of integral ideals $\mathfrak{a}\subseteq \mathcal{O}_K$ of norm $m$ in $\mathcal{O}_K$, that is,
 \begin{align*}
     f(m):=\sharp\{ \mathfrak{a}\subseteq \mathcal{O}_K: \mathcal{N}(\mathfrak{a})=m \}.
 \end{align*}
 Then we know that 
 \begin{align*}
     f(m)\leq d(m)^2,
 \end{align*}
where $d(m)$ is the number of divisors of $m\in \mathbb{N}$ (for a proof, see \cite[Equation~(68)]{C-N}). 
Moreover, it is well-known that 
 \begin{align*}
     d(m)\ll m^{\delta/\log\log m},
 \end{align*}
 for some constant $\delta>0$ (see \cite[Theorem 317]{HaWr}). It follows that 
 \begin{align*}
     r_Q(n)\le 6f(4a_1n)\le 6d(4a_1n)^2\ll_Q n^{c/\log\log n},
 \end{align*}
where $c=2\delta>0$. This completes the proof. 
\end{proof}
We are now ready to prove Corollary \ref{main_thm 2}. 
Let an irrational number $\alpha$ be given. Then by Lemma~\ref{lem_Dirichlet_approx}, there are infinitely many pairs $(a, q)\in \mathbb{N}\times \mathbb{Z}$ such that
\begin{align}\label{con_q_1}
    \left|\alpha-\frac{a}{q}\right|\le \frac{24 \Delta^2}{q^2}, \ (2a\Delta,q)=1.
\end{align}
This proves the last part of Corollary \ref{main_thm 2}.
Now, fix such a pair $(a,q)$ and suppose $L, X$ to be parameters satisfying
\begin{align*}
    L=q^\beta, \ X=Lq=q^{1+\beta},
\end{align*}
with $2/5+3\varepsilon\leq \beta<1.$ Recall the definition of $\mathcal{A}_Q$ from \eqref{def_A_Q}. If $n\in\mathcal{A}_Q$ with $n\le 2X$ satisfies one of the following congruences  
\begin{align*}
    an\equiv b \bmod{q}, \ |b|\leq L,\ (b,q)=1,
\end{align*}
then writing $an=b+n_1q$ for some integer $n_1$, it follows from \eqref{con_q_1} that
\begin{align*}
    \alpha n=n_1+\frac{b}{q}+\frac{\theta n}{q^2},
\end{align*}
where $|\theta|\le 24\Delta^2$. In other words, 
\begin{equation*} 
    \|\alpha n\|\ll_Q \frac{L}{q}+\frac{X}{q^2}=2q^{\beta-1}= 2X^{-(1-\beta)/(1+\beta)}=2X^{-\gamma} \ll n^{-\gamma}.
\end{equation*}
Hence, such $n$ lies in the set
\begin{align*}
    \{n\in \mathcal{A}_Q: n\le 2X, \ \|n\alpha\|<C_1n^{-\gamma}\}=:\mathcal{A}_Q(\alpha, X), \ \text{ say},
\end{align*}
where $C_1>0$ is a suitable constant depending on $Q$.
Thus, recalling
\begin{align*}
    r_Q(n)\ll n^{c/(\log\log n)} \quad \mbox{for some constant } c>0
\end{align*}
 from Lemma~\ref{lem_bound_r_Q(n)}, and noting that
\begin{align*}
    b_1\not\equiv b_2( \bmod{q} )\Longrightarrow n_1\not\equiv n_2 (\bmod{q})
\end{align*} 
if $q$ is large enough, 
we deduce from Theorem~\ref{main_thm} that
\begin{align*}
    \frac{XL}{q \log^2 q}\ll S=\sum\limits_{\substack{|b|\le L\\(b,q)=1}} \sum\limits_{\substack{n\in\mathbb{Z}\\ na\equiv b\bmod{q}}} r_Q(n)w(n)\ll X^{c/(\log\log X)} |A_Q(\alpha, X)|,
\end{align*}
that is,
\begin{align*}
    |\mathcal{A}_{Q}(\alpha, X)|\gg \frac{X^{1-\gamma-c/(\log\log X)}}{\log^2 q}.
\end{align*}
Since $q=X^{1/(1+\beta)}$, it follows that 
\begin{align*}
    |\mathcal{A}_{Q}(\alpha, X)|\gg X^{1-\gamma-C_2/(\log\log X)}
\end{align*}
if $C_2>c$. This completes the proof of Corollary \ref{main_thm 2}.
\section{Application of Voronoi summation}
Our proof of Theorem \ref{main_thm} begins with an application of Voronoi summation. We first detect the congruence condition on the right-hand side of \eqref{Sdef} using the
orthogonal property of additive characters, thus obtaining  
\begin{align*}
    S= & \frac{1}{q}\sum_{\substack{|b|\le L\\(b,q)=1}} \sum_{n\in\mathbb{Z}}r_Q(n)w(n) \sum_{g \bmod q}e_q\left((n-\overline{a}b)g\right)\\ = &  \frac{1}{q}\sum_{\substack{|b|\le L\\(b,q)=1}} \sum_{n\in\mathbb{Z}}r_Q(n)w(n)\sum_{k| q}\  \sideset{}{^*}\sum_{h \bmod k}e_k\left((n-\overline{a}b)h\right),
\end{align*}
where $a\overline{a}\equiv 1 \bmod q$, and the asterisk above indicates that $(h,k)=1$.  As the weight function $w(n)$ is compactly supported, all the summations in $S$ run over finite ranges and so we can interchange the order of  summations. Hence, 
\begin{align*}
    S=\frac{1}{q}\sum_{\substack{|b|\le L\\(b,q)=1}}\sum_{k| q}\ \sideset{}{^*}\sum_{h \bmod k}e_k(-\overline{a}bh)\sum_{n\in\mathbb{Z}}r_Q(n)w(n)e_k(nh).
\end{align*}
Now applying Lemma~\ref{Voronoi} to the inner-most sum, we get
\begin{equation}\label{S=T_1+T_2}
\begin{split}
    S= & \frac{2\pi}{q\Delta^{1/2}}\left(\int\limits_0^{\infty} w(x)dx\right)\sum_{\substack{|b|\le L\\(b,q)=1}}\sum_{k| q}\frac{1}{k^2} \sideset{}{^*}\sum_{h\bmod k}G_Q(k,h)e_k(-\overline{a}bh)+\\ & \frac{2\pi}{q\Delta^{1/2}}\sum_{\substack{|b|\le L\\(b,q)=1}}\sum_{k| q}\ \sideset{}{^*}\sum_{h\bmod k}\hspace{-0.3cm}e_k(-\overline{a}bh)\sum_{n\in \mathbb{Z}}\tilde{r}_{Q^*}(n,k,h)e_{k}(-\overline{h}\overline{\Delta}n)\tilde{w}_k(n)\\
    =: & T_1+T_2\ \text( say ),
\end{split}
\end{equation}
where \begin{align}\label{def_w_k(n)}
    \tilde{w}_k(n):=\frac{1}{k}\int\limits_0^{\infty} J_0\left(4\pi \frac{\sqrt{nx/\Delta}}{k}\right)w(x)dx .
\end{align}
Recalling our assumption $(2\Delta h,q)=1$, and using \eqref{def_G(k,h)}, \eqref{formula_G(k,h)} and \eqref{def_tilde{r}}, we obtain
\begin{equation}\label{tilde_r_2}
\begin{split}
    \tilde{r}_{Q^*}(n,k,h)&=k^{-1}\sum_{\substack{Q^*(\vec{x})=n}}\ \sum_{\vec{y}\bmod k}e_k(hQ(\vec{y}))\\
    &=k^{-1}\sum_{\substack{Q^*(\vec{x})=n}}G_Q(k,h)\\
    &=k^{-1}\sum_{\substack{Q^*(\vec{x})=n}}\left(\frac{h}{N(k,Q)}\right)G_Q(k,1)\\
    &=k^{-1}G_Q(k,1)r_{Q*}(n)\left(\frac{h}{N(k,Q)}\right).
     \end{split}
\end{equation}
\begin{lem}\label{lem_N(k,Q)}
Assume that $(2\Delta,q)=1$. Then for $k|q$, we have 
    \begin{align}
        N(k,Q)=1.
    \end{align}
\end{lem}
\begin{proof}
    Recall that 
    \begin{align*}
        N(k, Q)=\sharp\{\vec{x} \bmod{k} : A\vec{x}\equiv 0\bmod{k}\},
    \end{align*}
    where $Q=\frac{1}{2}(x,y)A(x,y)^T$. Now as $Q(x,y)=a_1x^2+b_1xy+c_1y^2$, we have 
    \begin{align*}
        A\vec{x}\equiv 0\bmod{k} \Longleftrightarrow \begin{cases} 2a_1x+b_1y\equiv 0 \bmod{k}\\ 
 b_1x+2c_1y\equiv 0\bmod{k}, \end{cases}
    \end{align*}
    where $\vec{x}=(x,y)^T$. Observe that the determinant of the coefficient matrix 
$$
\begin{pmatrix}
2a_1 & b_1\\
b_1 & 2c_1
\end{pmatrix}
$$
equals $\Delta$. Since  $(\Delta,k)=1$, it follows that the above system  
of congruences has only the trivial solution $x=y=0$ and hence $N(k,Q)=1$.  
\end{proof}
Combining the above Lemma~\ref{lem_N(k,Q)} with \eqref{tilde_r_2}, we see that 
\begin{align} \label{rform}
    \tilde{r}_{Q^*}(n,k,h)=k^{-1}G_Q(k,1)r_{Q*}(n).
\end{align}
\section{Evaluation of the main term $T_1$}
Using $k|q$ and Lemmas \ref{Lem_G_Q(k,h,a)} and \ref{lem_N(k,Q)},  we have  
\begin{align}\label{G(q,h)=G(q,1)}
    G_Q(k,h)=G_Q(k,1).
\end{align}
Therefore, the main term $T_1$ in \eqref{S=T_1+T_2} turns into 
\begin{equation}\label{MT}
\begin{split}
T_1= &   \frac{2\pi}{q\Delta^{1/2}}\left(\int\limits_0^{\infty} w(x)dx\right)\sum_{\substack{|b|\leq L\\(b,q)=1}}\sum_{k| q}\frac{1}{k^2} G_Q(k,1)\sideset{}{^*}\sum_{h\bmod k}e_k(-\overline{a}bh)\\
= & \frac{2\pi}{q\Delta^{1/2}}\left(\int\limits_0^{\infty} w(x)dx\right)\sum_{\substack{|b|\leq L\\(b,q)=1}}\sum_{k| q}\frac{\mu(k)}{k^2} G_Q(k,1)\\
= & \frac{2\pi}{q\Delta^{1/2}}\left(\int\limits_0^{\infty} w(x)dx\right)\left(\frac{\phi(q)}{q}\cdot 2L+O(\tau(q))\right)\sum_{k| q}\frac{\mu(k)}{k^2} G_Q(k,1),
\end{split}
\end{equation}
where $\tau(q)$ denotes the number of divisors of $q$. Here we have used the fact that the Ramanujan sum   
\begin{align*}
    c_k(\overline{a}b)=\sideset{}{^*}\sum_{h\bmod k}e_k(-\overline{a}bh)
\end{align*}
satisfies $c_k(\overline{a}b)=\mu(k)$ if $(\overline{a}b, k)=1$.
\begin{lem}\label{lem_G(k,1)_mul} The function $G_Q(k,1)$ is multiplicative on the set of odd positive integers, that is, for odd $k=k_1k_2$ with $(k_1, k_2)=1$, we have  
    \begin{align*}
        G_Q(k,1)=G_Q(k_1,1)G_Q(k_2,1).
    \end{align*}
\end{lem}
\begin{proof}
    Indeed, we calulate that
    \begin{align*}
        G_Q(k,1)
        =\sum_{x,y \bmod{k}} & e_k(a_1x^2+b_1xy+c_1y^2)\\
        =\sum_{x_1,y_1\bmod{k_1}}& \ \sum_{x_2,y_2\bmod{k_2}}e_k\biggl(a_1(x_1k_2+x_2k_1)^2+\\&b_1(x_1k_2+x_2k_1)(y_1k_2+y_2k_1)+c_1(y_1k_2+y_2k_1)^2\biggr)\\
        =\sum_{x_1,y_1\bmod{k_1}} & \ \sum_{x_2,y_2\bmod{k_2}}e\biggl(\frac{a_1k_2x_1^2}{k_1}+\frac{a_1k_1x_2^2}{k_2}+\frac{b_1k_2x_1y_1}{k_1}+\\
        &\frac{b_1k_1x_2y_2}{k_2}+\frac{c_1k_2y_1^2}{k_1}+\frac{c_1k_1y_2^2}{k_2}\biggr)\\
        =\sum_{x_1,y_1\bmod{k_1}} & e_{k_1}\biggl(k_2(a_1x_1^2+b_1x_1y_1+c_1y_1^2)\biggr)\times\\
         \sum_{x_2,y_2\bmod{k_2}} & e_{k_2}\biggl(k_1(a_1x_2^2+b_1x_2y_2+c_1y_2^2)\biggr)\\
        =G_Q(k_1,k_2)& G_Q(k_2,k_1)
        =G_Q(k_1,1)G_Q(k_2,1),
    \end{align*}
    where the last line is due to \eqref{G(q,h)=G(q,1)}.
\end{proof}
Now, by Lemma~\ref{lem_G(k,1)_mul},
\begin{equation} \label{lowupp}
    \sum_{k| q}\frac{\mu(k)}{k^2} G_Q(k,1)=\prod_{p|q}\left(1-\frac{G_Q(p,1)}{p^2}\right),
\end{equation}
and by Lemma~\ref{lem_bound_G_q(q,h)},
\begin{align}\label{inequality_G(k,1)}
   \prod_{p|q}\left(1-\frac{1}{p}\right)\le \left|\prod_{p|q}\left(1-\frac{G_Q(p,1)}{p^2}\right)\right|\le \prod_{p|q}\left(1+\frac{1}{p}\right),
\end{align}
where the products are restricted to prime divisors $p$ of $q$. This, together with \eqref{MT} gives 
\begin{align}\label{final_T_1}
    T_1=  \frac{4\pi \varphi(q)L}{\Delta^{1/2}q^2}\prod_{\substack{p|q\\p \text{ prime }}}\left(1-\frac{G_Q(p,1)}{p^2}\right) \int\limits_{0}^{\infty}w(x)dx + O_{\varepsilon}\left( \frac{Xq^{\varepsilon}}{q}\right),
\end{align}
where we have used the divisor bound $\tau(q)\ll_\varepsilon q^{\varepsilon}$.
Taking the well-known bound
$$
\frac{\varphi(q)}{q}=\prod\limits_{p|q}\left(1-\frac{1}{p}\right)\gg \frac{1}{\log\log q},
$$
into account, we deduce the following lemma from \eqref{inequality_G(k,1)} and \eqref{final_T_1}.
 
\begin{lem}\label{lem_MT} We have
\begin{align*}
    \frac{q^{2\beta}}{(\log\log q)^2}=\frac{XL}{q(\log\log q)^2}\ll T_1\ll \frac{XL}{q}=q^{2\beta}.
\end{align*}
\end{lem}

\section{Evaluation of $T_2$}

Inserting the formula \eqref{rform} for $\tilde{r}_{Q^*}(n,k,h)$ into \eqref{S=T_1+T_2}, we obtain
\begin{equation}\label{T_2_before_cut}
\begin{split}
    T_2&=\frac{2\pi}{q\Delta^{1/2}}\sum_{\st{b\leq L\\(b,q)=1}}\sum_{k| q}k^{-1}G_Q(k,1)\sum_{n\in \mathbb{Z}}r_{Q^*}(n)\tilde{w}_k(n)\sideset{}{^*}\sum_{h\bmod k}e_k(-\overline{a}bh-\overline{\Delta}\overline{h}n)\\
    &=\frac{2\pi}{q\Delta^{1/2}}\sum_{k| q}\frac{G_Q(k,1)}{k}\sum_{n\in \mathbb{Z}}\sum_{\st{b\leq L\\(b,q)=1}}r_{Q^*}(n)\tilde{w}_k(n)S(\overline{\Delta}n,\overline{a}b; k)\\
    &=\frac{2\pi}{q\Delta^{1/2}}\sum_{k| q}\frac{G_Q(k,1)}{k}\sum_{n\in \mathbb{Z}}\sum_{\st{b\leq L\\(b,q)=1}}r_{Q^*}(n)\tilde{w}_k(n)S(n,\overline{\Delta}\overline{a}b; k),
\end{split}
\end{equation}
since $(\overline{\Delta}, k)=1$. 

In order to bound $\tilde{w}_k(n)$, we use the following facts about Bessel functions, which are available at many places (see \cite[equations (9) and (10)]{Blomer}, for example). For any $\beta>0$ and positive integer $j$, by repeated integrations by parts, we have
\begin{align}\label{w^j(x)}
    \int\limits_0^{\infty}{w(x) J_0(\beta \sqrt{x}) dx }= \left(\frac{2}{\beta}\right)^j\int\limits_0^{\infty} w^{(j)}(x)x^{j/2}J_j(\beta\sqrt{x})dx.
\end{align}
Also, for any non-negative integer $j$ and arbitrarily small $\varepsilon>0$, we have 
\begin{align}\label{bound_J_j(x)}
    J_j(x)\ll_{\varepsilon} \frac{1}{\sqrt{1+x}}+\frac{1}{x^{j-\varepsilon}}.
\end{align}
Using the trivial bound $J_j(x)\ll_{\varepsilon} x^{\varepsilon}$ together with the fact that $w^{(j)}(x)\ll x^{-j}$, and recalling that $w(x)$ has compact support in $[X,2X]$, we infer from \eqref{def_w_k(n)} and \eqref{w^j(x)} that
\begin{equation}\label{bound_w(n)_in_j}
\begin{split}
    \tilde{w}_k(n) & \ll_{j,\varepsilon} \frac{1}{k}\left(\frac{k}{\sqrt{n}}\right)^j\int\limits_X^{2X}x^{-j}x^{j/2}x^{\varepsilon} dx \\
   & \ll_{j,\varepsilon}  \frac{1}{k}\left(\frac{k}{\sqrt{n}}\right)^jX^{-j/2 +1+\varepsilon}\\
    & =  \frac{X^{1+\varepsilon}}{k}\left(\frac{k}{\sqrt{Xn}}\right)^j
\end{split}
\end{equation}
for any $j\geq 1$. 
Thus, taking $j$ sufficiently large, we see that $\tilde{w}_k(n)$ is negligible (that is, $\tilde{w}_k(n)\ll X^{-20260}$) unless 
\begin{align*}
    \frac{k}{\sqrt{Xn}}\geq X^{-\varepsilon/2}, \text{  that is, } n\leq \frac{k^2}{X^{1-\varepsilon}}.
\end{align*} 
Hence, it follows from \eqref{T_2_before_cut} that 
\begin{align}\label{T_2_cut}
    T_2&=\frac{2\pi}{q\Delta^{1/2}}\sum_{k| q}\frac{G_Q(k,1)}{k}\sum_{n\leq k^2/X^{1-\varepsilon}} \sum_{\st{b\leq L\\(b,q)=1}}r_{Q^*}(n)\tilde{w}_k(n)S(n,\overline{\Delta}\overline{a}b; k)+O(X^{-2026}). 
\end{align}
\begin{lem}\label{lem_bound_gamma}
Set 
$$
\gamma(n):= \begin{cases} r_{Q^*}(n)\tilde{w}_k(n) & \mbox{ if } 1\le n\le k^2/X^{1-\varepsilon},\\ 0 & \mbox{ otherwise,} \end{cases}  
$$
where $r_{Q^*}(n)$ and $\tilde{w}_k(n)$ are given as in $\eqref{T_2_cut}$. Then we have 
\begin{align*}
    \|\bm{\gamma}\|_2:=\left(\sum_{n\in \mathbb{Z}}|\gamma(n)|^2\right)^{1/2}\ll_{\varepsilon} (kX)^{2\varepsilon}X^{1/2}.
\end{align*}
\end{lem}
\begin{proof}
    Taking $j=0$ in \eqref{bound_w(n)_in_j}, we get  
    \begin{align*}
        |\tilde{w}_k(n)|\le \frac{X^{1+\varepsilon}}{k}.
    \end{align*}
Since $|r_{Q^*}(n)|\ll_{\varepsilon} n^{\varepsilon}$ for any $\varepsilon>0$, it follows that 
\begin{align*}
    \|\bm{\gamma}\|_2\ll (kX)^{2\varepsilon}X^{1/2}.
\end{align*}
\end{proof}
Now we are ready to bound the term $T_2$ in the following lemma.

\begin{lem}\label{lem_ET}
   Let $T_2$ be as given in \eqref{T_2_before_cut}. Then we have 
    \begin{align*}
        T_2\ll_{\varepsilon} q^{\varepsilon}(X^{1/2}L^{1/4}+X^{1/2}q^{-1/2}L+q^{1/2}L^{1/2})
    \end{align*}
for any $\varepsilon>0$. 
\end{lem}
\begin{proof}
In order to apply Lemma~\ref{lem_bound_aver_Kloos}, we need to remove the coprimality condition $(b,q)=1$ in the innermost sum. We do this by M\"obius inversion, obtaining
\begin{align*}
    &\sum_{n\leq k^2/X^{1-\varepsilon}} \sum_{\st{b\leq L\\(b,q)=1}}r_{Q^*}(n)\tilde{w}_k(n)S(n,\overline{\Delta}\overline{a}b; k)\\
    =&\sum_{n\leq k^2/X^{1-\varepsilon}} \sum_{\st{b\leq L}}\left( \sum_{d|(b,q)}\mu(d)\right)r_{Q^*}(n)\tilde{w}_k(n)S(n,\overline{\Delta}\overline{a}b; k)\\
=&\sum_{d|q}\mu(d)\sum_{n\leq k^2/X^{1-\varepsilon}} \sum_{\st{b\leq L/d}}r_{Q^*}(n)\tilde{w}_k(n)S(n,\overline{\Delta}\overline{a}db; k).  
\end{align*}
Set  $s:=(\overline{\Delta}\overline{a}d,k)=(d,k)$. Then
applying Lemma \ref{lem_bound_aver_Kloos} with $\gamma(n)=r_{Q^{*}}(n)\tilde{w}_k(n)$ to the bilinear sum over $n$ and $b$ in the last line and choosing $S_1(M, N, q,d)$ as in \eqref{relevantbound}, we obtain
\begin{align*}
&\sum_{n\leq k^2/X^{1-\varepsilon}} \sum_{\st{b\leq L/d}}r_{Q^*}(n)\tilde{w}_k(n)S(n,\overline{\Delta}\overline{a}db; k)\\
\ll & X^{\varepsilon}\|\bm{\gamma}\|_2(k^2/X)^{1/2}(L/d) k^{1/2+o(1)}\left\{(k^2/X)^{-1/2}\left((L/d)^{-3/4}k^{1/2}+s^{1/2}\right)+(L/d)^{-1/2}\right\}\\
\ll & (kX)^{\varepsilon}\|\bm{\gamma}\|_2\left\{kL^{1/4}d^{-1/4}+k^{1/2}Ld^{-1/2}+k^{3/2}L^{1/2}d^{-1/2}X^{-1/2}\right\}\\
\ll & (kX)^{3\varepsilon}\left\{X^{1/2}kL^{1/4}d^{-1/4}+X^{1/2}k^{1/2}Ld^{-1/2}+k^{3/2}L^{1/2}d^{-1/2}\right\},
\end{align*}
where the last line is due to the bound for $\|\bm{\gamma}\|_2$ in Lemma~\ref{lem_bound_gamma}.
Therefore,
\begin{align*}
    &\sum_{d|k}\mu(d)\sum_{n\leq k^2/X^{1-\varepsilon}} \sum_{\st{b\leq L/d}}r_{Q^*}(n)\tilde{w}_k(n)S(n,\overline{\Delta}\overline{a}db; k)\\
    \ll &(kX)^{3\varepsilon}k^{\varepsilon/2}\left\{X^{1/2}kL^{1/4}+X^{1/2}k^{1/2}L+k^{3/2}L^{1/2}\right\}.
\end{align*}
Combining this with \eqref{T_2_cut} and using the bound $|G_Q(k,1)|\le k$ from Lemma~\ref{lem_bound_G_q(q,h)}, we obtain
\begin{align*}
    T_2 \ll (Xq)^{4\varepsilon}(X^{1/2}L^{1/4}+X^{1/2}q^{-1/2}L+q^{1/2}L^{1/2}).
\end{align*}
As $X$ is a power of $q$, changing $\varepsilon$ in a suitable way, the bound in Lemma \ref{lem_ET} follows.
\end{proof}
\section{Proof of Theorem~\ref{main_thm}}
By \eqref{S=T_1+T_2}, we have 
\begin{align*}
    S=T_1+T_2,
\end{align*}
where $T_1$ is the main term and $T_2$ the error term. The main term satisfies the asymptotic formula \eqref{final_T_1} and the bounds
\begin{align*}
    \frac{q^{2\beta}}{(\log\log q)^2}\ll T_1\ll q^{2\beta}
\end{align*}
from Lemma~\ref{lem_MT}. Lemma~\ref{lem_ET} provides the upper bound 
\begin{align*}
    T_2\ll q^{\varepsilon}(X^{1/2}L^{1/4}+X^{1/2}q^{-1/2}L+q^{1/2}L^{1/2})
\end{align*}
for the error term.
Recall the choices $L=q^\beta$ and $X=qL=q^{1+\beta}$ with $0<\beta<1.$ We calculate that 
\begin{align}\label{final_T_2}
    T_2= O_\varepsilon\left(q^{2\beta-\varepsilon}\right)
\end{align}
if $\beta\ge 2/5+3\varepsilon$. This implies the result in Theorem \ref{main_thm}.

\begin{re}
We point out that the Weil bound $S(m, an; q)\ll (m,n,q)^{1/2}q^{1/2+o(1)}$ for the Kloosterman sum yields the ``trivial bound" 
$$
S_{q,a}(\bm{\alpha})\ll \left(\sum_{m\in I}|\alpha_m|^2\right)^{1/2}M^{1/2}Nq^{1/2+o(1)}
$$
in place of \eqref{bibound}. Using this bound leads to a shorter $\beta$-range $1/2<\beta\le 1$ in place of $2/5<\beta\le 1$ in Corollary \ref{main_thm 2} and thus restricts the exponent $\gamma$ to the shorter interval $0\le \gamma<1/3$ in place of $0\le \gamma<3/7$. Conjecturally, the best we may hope for is the bound 
$$
S_{q,a}(\bm{\alpha})\ll \left(\sum_{m\in I}|\alpha_m|^2\right)^{1/2}M^{1/2}N^{1/2}q^{1/2+o(1)}
$$  
(square root cancellation on average in the $n$-sum), which would result in a $\beta$-range of $1/3<\beta\le 1$ and hence in a $\gamma$-interval of $0\le \gamma<1/2$. Extending these ranges further would likely require different techniques than those used in the present paper.
\end{re}





\section*{Acknowledgments}
We thank Prof. Harman for bringing Cook's and Danicic's results in \cite{Cook} and \cite{Danicic} to our attention, which we had overlooked in an earlier version of this paper. The first-named author would like to thank the Ramakrishna Mission Vivekananda Educational and Research Insititute for an excellent work environment. The research of the second-named author was supported by a Prime Minister Research Fellowship (PMRF ID- 0501972), funded by the Ministry of Education, Govt. of India.


\end{document}